\documentclass[12pt]{article}

\usepackage{amsfonts,amsthm,amssymb,amsmath}
\usepackage{cite}
\usepackage{graphicx}
\usepackage[margin=1in]{geometry} 
\usepackage{float}
\usepackage{subcaption}
\usepackage{authblk}
\usepackage{multicol}
\usepackage{sidecap}
\sidecaptionvpos{figure}{b}

\newtheorem{theorem}{Theorem}
\newtheorem{corollary}[theorem]{Corollary}
\newtheorem{definition}[theorem]{Definition}
\newtheorem{example}[theorem]{Example}
\newtheorem{lemma}[theorem]{Lemma}
\newtheorem{remark}[theorem]{Remark}

\newcommand{\crd}{{\rm C}_{\rm rd}}
\providecommand{\keywords}[1]{\textbf{\textit{Keywords:}} #1}
\providecommand{\subjclass}[1]{\textbf{\textit{MSC 2010:}} #1}

\title{Exact Solution to a Dynamic SIR Model\thanks{This is a preprint 
of a paper whose final and definite form is with 
\emph{Nonlinear Analysis: Hybrid Systems}, ISSN: 1751-570X, 
available at \texttt{https://doi.org/10.1016/j.nahs.2018.12.005}.
Submitted 16/May/2018; Revised 10/Oct/2018; Accepted for publication 18/Dec/2018.}}

\author[1]{Martin Bohner\thanks{Martin Bohner: bohner@mst.edu}}
\affil[1]{\emph{Department of Mathematics \& Statistics,
Missouri University of Science and Technology, 65409 Rolla, MO, USA}}

\author[2]{Sabrina Streipert\thanks{Sabrina Streipert: s.streipert@uq.edu.au}}
\affil[2]{\emph{Centre for Applications in Natural Resource Mathematics,\newline
School of Mathematics and Physics, University of Queensland,\newline 
4067 St Lucia, QLD, Australia}}

\author[3]{Delfim F. M. Torres\thanks{Delfim F. M. Torres: delfim@ua.pt}}
\affil[3]{\emph{R\&D unit CIDMA, Department of Mathematics,
University of Aveiro, 3810-193 Aveiro, Portugal}}

\begin{document}
	
\date{}

\maketitle


\begin{abstract}
We investigate an epidemic model based on Bailey's continuous differential system. 
In the continuous time domain, we extend the classical model to time-dependent coefficients 
and present an alternative solution method to Gleissner's approach. 
If the coefficients are constant, both solution methods yield the same result. 
After a brief introduction to time scales, we formulate the SIR (susceptible-infected-removed) 
model in the general time domain and derive its solution. In the discrete case, this provides 
the solution to a new discrete epidemic system, which exhibits the same behavior 
as the continuous model. The last part is dedicated to the analysis of the limiting behavior 
of susceptible, infected, and removed, which contains biological relevance.
\end{abstract}





\section{Introduction}

Modeling infectious diseases is as important as it has been in 1760, when Daniel Bernoulli 
presented a solution to his mathematical model on smallpox. It was however not until the $20$th 
century that mathematical models became a recognized tool to study the causes and effects of epidemics. 
In 1927, Kermack and McKendrick introduced their SIR-model based on the idea of grouping 
the population into susceptible, infected, and removed. The model assumes a constant total 
population and an interaction between the groups determined by the disease transmission 
and removal rates. Although the removed represent in some models the vaccinated individuals, 
it can also be used to transform a time dependent population size into a constant population. 
In the latter case, the total number of contacts that a susceptible individual could get 
in contact with, is not the individuals of all three groups but $x+y$, where $x$ is the number 
of susceptible and $y$ the number of infected individuals. Let $k$ be the actual number 
of individuals a susceptible interacts with and $p$ be the probability that a susceptible 
gets infected at contact with an infected individual. Then $pk \frac{y}{x+y}$ is the rate 
at which one susceptible enters the group of infected \cite{MR0452809}. This leads 
to the rate of change for the group of susceptible as
\begin{equation*}
\frac{{\rm d} x}{{\rm d} t}=-pk\frac{y}{x+y}x.
\end{equation*}
Similarly, the infected increase by that rate, but some infected leave the class of infected, 
due to death for example, at a rate $c$, which yields the differential equation in $y$ as
\begin{equation*}
\frac{{\rm d} y}{{\rm d} t}=pk\frac{y}{x+y}x-cy.
\end{equation*}
To obtain a system with time independent sum, a third group is added, 
the group of removed individuals for example, denoted by $z$, 
with the dynamics given by
\begin{equation*}
\frac{{\rm d} z}{{\rm d} t}=cy.
\end{equation*}
Many modifications of the classical model have been 
investigated such as models with vital dynamics, see \cite{MR1706068,MR1229446,MR3703345}. 
To model the spreading of diseases between different states, a spatial variable was added, 
which led to a partial differential system, see \cite{MR609362,MR2311671}. Already in 1975, 
Bailey discussed in \cite{MR0452809} the relevance of stochastic terms in the mathematical 
model of epidemics, which is still an attractive way of modeling the uncertainty 
of the transmission and vaccines, see \cite{MR3648939,MR2877694,MR3648941,MR3808514}. 
Although these modifications exist, so far there has been no success in generalizing 
the epidemic models to a general time scale to allow modeling a noncontinuous disease dynamics. 
A disease, where the virus remains within the host for several years unnoticed before 
continuing to spread, is only one example that can be modeled by time scales. We trust 
that this work provides the foundation for further research on generalizing epidemic 
models to allow modeling of discontinuous epidemic behavior. 


\section{Continuous SIR Model}

We investigate a susceptible-infected-removed (SIR) model 
proposed by Norman Bailey in \cite{MR0452809} of the form
\begin{equation}
\label{Bailey1}
\begin{cases} x' = -\frac{b x y}{x+y}, & \quad \\
y' =\frac{b x y}{x+y}-cy, & \quad \\
z'= cy, & \quad 
\end{cases}
\end{equation}
with initial conditions $x(t_0)=x_0>0$, $y(t_0)=y_0>0$, $z(t_0)=z_0 \geq 0$, 
$x,y,z: \mathbb{R}\to \mathbb{R}_0^+$, and $b,c \in \mathbb{R}_0^+$.
The variable $x$ represents the group of susceptible, $y$ the infected population, 
and $z$ the removed population. By adding the group of removed, 
the total population $N=x+y+z$ remains constant.
In \cite{MR945530}, assuming $x,y>0$, the model is solved by rewriting 
the first two equations in \eqref{Bailey1} as
\begin{equation*}
\begin{cases}
\frac{x'}{x} = -\frac{b }{x+y}y, & \quad \\
\frac{y'}{y} =\frac{b x }{x+y}-c. & \quad\\
\end{cases}
\end{equation*}
Subtracting these equations yields
\begin{equation*}
\frac{x'}{x}-\frac{y'}{y}=-b+c,
\end{equation*}
i.e., 
\begin{equation*}
\frac{y'}{y}=\frac{x'}{x}+b-c,
\end{equation*}
which is equivalent to
\begin{equation*}
(\ln y )' = (\ln x)'+(b-c).
\end{equation*}
Integrating both sides and taking the exponential, one gets
\begin{equation}
\label{overlap}
y  = x \kappa e^{(b-c)(t-t_0)}, 
\quad \mbox{ where } \kappa = \frac{y_0}{x_0}.
\end{equation}
If $b \neq c$, then, plugging this into the first equation in \eqref{Bailey1} 
yields a first order linear homogeneous differential equation 
with the solution given by
\begin{equation}
\label{Eq1_B}
x(t)=x_0 \left( 1+\kappa \right)^{\frac{b}{b-c}} 
\left( 1+\kappa e^{(b-c)(t-t_0)} \right)^{-\frac{b}{b-c}}.
\end{equation}
Replugging yields the solution of \eqref{Bailey1} as
\begin{equation}
\label{sol:gen1}
\begin{cases}
x(t)=x_0 (1+\kappa)^{\frac{b}{b-c}}(1+\kappa e^{(b-c)(t-t_0)})^{-\frac{b}{b-c}},\\
y(t)=y_0 (1+\kappa)^{\frac{b}{b-c}}(1+\kappa e^{(b-c)(t-t_0)})^{-\frac{b}{b-c}}e^{(b-c)(t-t_0)}, \\
z(t)=N-(x_0+y_0)^{\frac{b}{b-c}}(x_0+y_0 e^{(b-c)(t-t_0)})^{-\frac{c}{b-c}}.
\end{cases}
\end{equation}
If $b=c$, then \eqref{overlap} gives $y=x \kappa$, 
and the solution \eqref{sol:gen1} of \eqref{Bailey1} is
\begin{equation*}
\begin{cases}
x(t)=x_0 e^{-\frac{b\kappa (t-t_0)}{1+\kappa}},& \quad \\
y(t)=y_0 e^{-\frac{b\kappa (t-t_0)}{1+\kappa}}, & \quad \\
z(t)=N-(x_0+y_0)e^{-\frac{b\kappa (t-t_0)}{1+\kappa}}.& \quad \\
\end{cases}
\end{equation*}

In this work, we present a different method to solve \eqref{Bailey1}, 
considering not only constant $b,c$ but $b,c:\mathbb{R} \to \mathbb{R}^+$. 
This will allow us to find the solution to the model on time scales. 
To this end, define $w:=\frac{x}{y}>0$ for $x,y >0$ to get 
\begin{equation*}
w'=\frac{x'y-y'x}{y^2}=\frac{\frac{-b(t)xy}{x+y}y 
- \left( \frac{b(t)xy}{x+y}-c(t)y\right)x}{y^2}=\frac{-b(t)xy + c(t)xy}{y^2}=(c-b)(t)w,
\end{equation*}
which is a first-order homogeneous differential equation with solution 
\begin{equation*}
w(t)=w_0e^{\int_{t_0}^t (c-b)(s) \, {\rm d} s},
\end{equation*}
i.e., 
\begin{equation}
\label{solpart1}
y(t)=\kappa e^{\int_{t_0}^t (b-c)(s) \, {\rm d} s}x(t),
\end{equation}
which is the same as \eqref{overlap} for constant $b,c$. 
We plug \eqref{solpart1} into \eqref{Bailey1} to get
\begin{equation*}
x'=-\frac{b(t)x^2\kappa e^{\int_{t_0}^t (b-c)(s) 
\, {\rm d} s}}{x+\kappa x e^{\int_{t_0}^t (b-c)(s) \, {\rm d} s}}
=-\frac{b(t)\kappa e^{\int_{t_0}^t (b-c)(s) 
\, {\rm d} s}}{1+\kappa e^{\int_{t_0}^t (b-c)(s) \, {\rm d} s}}x,
\end{equation*} 
which has the solution
\begin{equation}
\label{Eq2_B}
x(t)=x_0 \exp \left \{   -\kappa \int_{t_0}^t b(s)\left(  \kappa 
+ e^{ \int_{t_0}^s (c-b)(\tau)\, {\rm d} \tau} \right)^{-1}\, {\rm d} s  \right \}.
\end{equation}
Note that, for constant $b,c$ with $b\neq c$, \eqref{Eq2_B} simplifies to \eqref{Eq1_B}.
Hence, the solution to \eqref{Bailey1} is given by
\begin{equation}
\label{sol:Sadded}
\begin{cases}
x(t)=x_0 \exp \left \{ - \kappa \int_{t_0}^t b(s)\left(  \kappa 
+ e^{ \int_{t_0}^s (c-b)(\tau)\, {\rm d} \tau} \right)^{-1}
\, {\rm d} s  \right \},& \quad \\[2mm]
y(t)=y_0 \exp \left \{ \int_{t_0}^t  \left[b(s)  \left(  1 
+ \kappa e^{ \int_{t_0}^s (b-c)(\tau)\, {\rm d} \tau} \right)^{-1}
- c(s)\right]\, {\rm d} s  \right \},& \quad \\[2mm]
z(t)=N-\left( y_0 e^{\int_{t_0}^t (b-c)(s) \, {\rm d} s} 
+x_0\right)\exp \left \{ - \kappa \int_{t_0}^t b(s)\left(  \kappa 
+ e^{ \int_{t_0}^s (c-b)(\tau)\, {\rm d} \tau} \right)^{-1}\, 
{\rm d}s  \right \}.& \quad \\
\end{cases}
\end{equation}
The time-varying parameters $b$ and $c$ allow us to investigate epidemic models, 
where the transmission rate peaks in early years before reducing, for example 
due to initial ignorance but increasing precaution of susceptibles. This behavior 
could be described by the probability density function of the log-normal 
distribution. A removal rate that increases rapidly to a constant rate could 
be modeled by a ``von Bertalanffy'' type function, see Figure~\ref{Figexi1}.
\begin{SCfigure}
\centering
\includegraphics[width=0.5\textwidth]{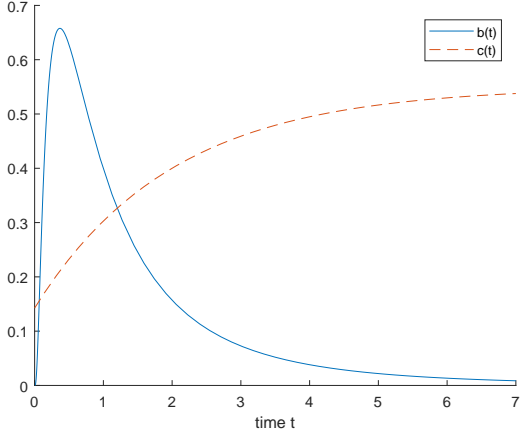}
\caption{Time-varying parameters $b$ and $c$ of ``von Bertalanffy'' type:	
$b(t) = \frac{1}{t \sqrt{2\pi}}e^{-\frac{\ln^2(t)}{2}}$ and
$c(t) = 0.55\left( 1-e^{-0.5t-0.3}  \right)$.}
\label{Figexi1}
\end{SCfigure}
Using these parameter functions with initial conditions $x_0=0.4$ and $y_0=1.2$ 
leads to the behavior in Figure~\ref{Figexi2}. We see that the group of infected 
increases before reducing due to an increasing removal rate $c$. 
Zooming into the last part of the time interval, we see that the number 
of susceptibles converges.
\begin{SCfigure}
\centering
\includegraphics[width=0.5\textwidth]{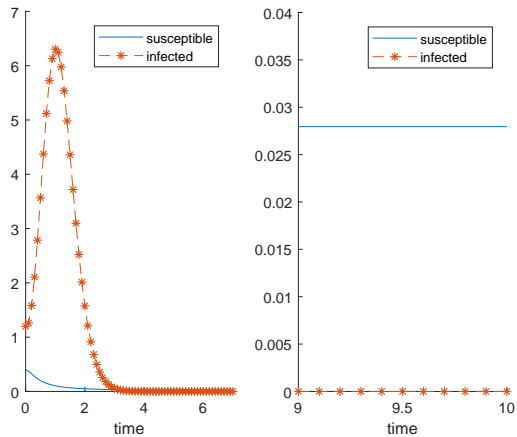}
\caption{Dynamics of the susceptible $x$ and the infected $y$ populations
with initial conditions $x_0=0.4$ and $y_0=1.2$ in the case of 
the time-varying coefficients $b$ and $c$ of Figure~\ref{Figexi1}.}
\label{Figexi2}
\end{SCfigure}

\begin{example}
Considering a simple decreasing transmission rate to account for the rising 
precaution of susceptibles and a simple decreasing removal rate accounting 
for medical advances, for example by choosing $b(t)=\frac{1}{t+1}$ 
and $c(t)=\frac{2}{t+1}$, the solution with $t_0=0$ is given by \eqref{sol:Sadded} as 
\begin{equation*}
\begin{cases}
x(t)=x_0 \exp \left\{  \int_{0}^t 
\frac{- \kappa}{ (s+1)(\kappa + s+1)}\, {\rm d} s  \right\}
=x_0 \frac{\kappa+1+t}{(\kappa+1)(t+1)}, &\\[2.5mm]
y(t)=y_0 \exp \left \{ \frac{1}{\kappa+1+s} -\frac{2}{s+1}  \right\}
=y_0 \frac{\kappa+1+t}{(\kappa+1)(t+1)^2}, &\\[2.5mm]
z(t)=N-\frac{1+\kappa+t}{t+1}\left\{  \frac{x_0}{\kappa+1}
-\frac{y_0}{(\kappa+1)(t+1)}\right\},&
\end{cases}
\end{equation*}
where $N=x_0+y_0+z_0$ and $\kappa=\frac{y_0}{x_0}$.
\end{example}


\section{Time Scales Essentials}

In order to formulate the time scales analogue to the model proposed by Norman Bailey, 
we first introduce fundamentals of time scales that we will use. The following 
introduces the main definitions in the theory of time scales. 

\begin{definition}[See \protect{\cite[Definition 1.1]{MR1843232}}]
For $t\in \mathbb{T}$, the forward jump operator $\sigma:\mathbb{T}\to \mathbb{T}$ is 
\begin{equation*}
\sigma(t):=\inf\{s \in \mathbb{T}:  s \, > t\}.
\end{equation*}
For any function $f:\mathbb{T} \to \mathbb{R}$, we put $f^{\sigma}=f \circ \sigma$. 
If $t \in \mathbb{T}$ has a left-scattered maximum $M$, then we define 
$\mathbb{T}^{\kappa} = \mathbb{T} \setminus \{M\}$; 
otherwise, $\mathbb{T}^{\kappa} = \mathbb{T}$.
\end{definition}

\begin{definition}[See \protect{\cite[Definition 1.24]{MR1962542}}]
A function $p:\mathbb{T} \to \mathbb{R}$ is called rd-continuous provided 
$p$ is continuous at $t$ for all right-dense points $t$ and the left-sided 
limit exists for all left-dense points $t$. The set of rd-continuous functions 
is denoted by ${\rm C}_{\rm rd} = {\rm C}_{\rm rd}(\mathbb{T}) 
= {\rm C}_{\rm rd}(\mathbb{T}, \mathbb{R})$. 
\end{definition}

\begin{definition}[See \protect{\cite[Definition 2.25]{MR1843232}}]
A function $p:\mathbb{T}\to \mathbb{R}$ is called regressive provided 
$$ 
1+\mu(t)p(t) \neq 0  \quad  \mbox{  for all  }  
t \in \mathbb{T}, \quad \mbox{where } \mu(t)=\sigma(t)-t. 
$$
The set of \emph{regressive} and \emph{rd-continuous functions} is denoted by $\mathcal{R}=\mathcal{R}(\mathbb{T})=\mathcal{R}(\mathbb{T},\mathbb{R})$. 
Moreover, $p \in \mathcal{R}$ is called positively regressive, 
denoted by $\mathcal{R^+}$, if 
$$
1+\mu(t)p(t) > 0 \quad \mbox{  for all  }   t \in \mathbb{T}. 
$$
\end{definition}

\begin{definition}[See \protect{\cite[Definition 1.10]{MR1843232}}]
Assume $f:\mathbb{T}\to \mathbb{R}$ and $t \in \mathbb{T}^{\kappa}$. 
Then the derivative of $f$ at $t$, denoted by $f^{\Delta}(t)$, 
is the number such that for all $\varepsilon >0$, there exists $\delta>0$, such that 
\begin{equation*}
\left | f(\sigma(t))-f(s)-f^{\Delta}(t)(\sigma(t)-s)\right| 
\leq \varepsilon \left | \sigma(t)-s\right |
\end{equation*}
for all $s \in (t-\delta,t+\delta) \cap \mathbb{T}$.
\end{definition}

\begin{theorem}[See \protect{\cite[Theorem 2.33]{MR1843232}}]
Let $p \in \mathcal{R}$ and $t_0 \in \mathbb{T}$. Then  
\begin{equation*}
y^{\Delta}=p(t)y, \quad y(t_0)=1
\end{equation*}
possesses a unique solution, called the exponential 
function and denoted by $e_p(\cdot,t_0)$.  
\end{theorem}

Useful properties of the exponential function are the following.

\begin{theorem}[See \protect{\cite[Theorem 2.36]{MR1843232}}]
\label{Theorem3}
If $p \in \mathcal{R}$, then  
\begin{enumerate}
\item $e_0 (t,s)=1$, and $e_p(t,t)=1$, 
\item $e_p(t,s)=\frac{1}{e_p(s,t)}$,
\item the semigroup property holds:  $e_p(t,r)e_p(r,s)=e_p(t,s)$.
\end{enumerate}
\end{theorem}

\begin{theorem}[See \protect{\cite[Theorem 2.44]{MR1843232}}]
If $p \in \mathcal{R}^{+}$ and $t_0 \in \mathbb{T}$, 
then $e_p(t,t_0) >0$ for all $t \in \mathbb{T}$.
\end{theorem}

We define a ``circle-plus'' and ``circle-minus'' operation.

\begin{definition}[See \protect{\cite[p.\ 13]{MR1962542}}]
Define the ``circle plus'' addition on $\mathcal{R}$ as  
\begin{equation*}
p \oplus q =p + q + \mu p q
\end{equation*}
and the ``circle minus'' subtraction as 
\begin{equation*}
p \ominus q = \frac{p-q}{1+\mu q}.
\end{equation*}
\end{definition}

It is not hard to show the following identities.

\begin{corollary}[See \protect{\cite{MR1843232}}]
If $p,q \in \mathcal{R}$, then 
\begin{itemize}
\item[a)] $e_{p\oplus q}(t,s)=e_{p}(t,s)e_{q}(t,s),$
\item[b)] $e_{\ominus p}(t,s)=e_{p}(s,t)=\frac{1}{e_{p}(t,s)}.$
\end{itemize}
\end{corollary}

\begin{theorem}[Variation of Constants, see \protect{\cite[Theorems 2.74 and 2.77]{MR1843232}}]
Suppose $p \in \mathcal{R}$ and $f \in \crd$. Let $t_0 \in \mathbb{T}$ and $y_0 \in \mathbb{R}$. 
The unique solution of the IVP
\begin{equation*}
y^{\Delta} = p(t)y+f(t), \quad y(t_0)=y_0
\end{equation*}
is given by 
\begin{equation*}
y(t)=e_p(t,t_0)y_0 + \int_{t_0}^{t} e_p(t,\sigma(s))f(s)\, \Delta s.
\end{equation*}
The unique solution of the IVP
\begin{equation*}
y^{\Delta} = -p(t)y^{\sigma}+f(t), \quad y(t_0)=y_0
\end{equation*}
is given by 
\begin{equation*}
y(t)=e_{\ominus p}(t,t_0)y_0 + \int_{t_0}^{t} e_{\ominus p}(t,s)f(s)\, \Delta s.
\end{equation*}
\end{theorem}

\begin{lemma}[See \protect{\cite[Theorem 2.39]{MR1843232}}]
If $p\in \mathcal{R}$ and $a, b, c \in \mathbb{T}$, then 
\begin{equation*}
\int_{a}^{b} p(t) e_{p}(t,c)\Delta t =e_{p}(b,c)-e_{p}(a,c)
\end{equation*}
and
\begin{equation*}
\int_{a}^{b} p(t) e_{p}(c,\sigma(t)) \Delta t=e_{p}(c,a)-e_{p}(c,b).
\end{equation*}
\end{lemma}


\section{Dynamic SIR Model}

In this section, we formulate a dynamic epidemic model based on Bailey's 
classical differential system \eqref{Bailey1} and derive its exact solution. 
In the special case of a discrete time domain, this provides a novel model 
as a discrete analogue of the continuous system. We end the discussion by 
analyzing the stability of the solutions to the dynamic model 
in the case of constant coefficients. 

Consider the dynamic susceptible-infected-removed model of the form
\begin{equation}
\label{SIRT1}
\begin{cases}
x^{\Delta}  = -\frac{b(t)xy^{\sigma}}{x+y},& \quad \\[2mm]
y^{\Delta}  = \frac{b(t)xy^{\sigma}}{x+y} - c(t)y^{\sigma},& \quad \\[2mm]
z^{\Delta}  = c(t)y^{\sigma},& \quad\\
x,y >0,&\\
\end{cases}
\end{equation}
with given initial conditions $x(t_0)=x_0>0$, $y(t_0)=y_0> 0$, $z(t_0)=z_0\geq 0$, 
$x, y: \mathbb{T} \to \mathbb{R}^+$, $z: \mathbb{T} \to \mathbb{R}_0^+$, 
and $b, c: \mathbb{T} \to \mathbb{R}_0^+$.

\begin{theorem}
\label{Thm1}
If $c-b,g \in \mathcal{R}$, then the unique solution 
to the IVP \eqref{SIRT1} is given by 
\begin{equation*}
\begin{cases}
x(t) = e_{\ominus g}(t,t_0)x_0,& \quad \\
y(t) = e_{\ominus (g \oplus (c-b))}(t,t_0)y_0,& \quad \\
z(t) = N - e_{\ominus g}(t,t_0) \left( x_0 
+ y_0 e_{\ominus (c-b)}(t,t_0) \right),\\
\end{cases}
\end{equation*}
where $N=x_0+y_0+z_0$, $\kappa= \frac{y_0}{x_0}$, and 
\begin{equation*}
g(t) := \frac{b(t)\kappa }{\kappa (1+\mu(t)(c-b)(t))+ e_{c-b}(\sigma(t),t_0)}.
\end{equation*}
\end{theorem}

\begin{proof}
Assume that $x,y,z$ solve \eqref{SIRT1}. Since $(x+y+z)^{\Delta}=0$, 
we get $z=N - (x+y)$, where $N=x_0+y_0+z_0$. 
Defining $w := \frac{x}{y}$, we have
\begin{equation*}
w^{\Delta}=\frac{x^{\Delta} y - y^{\Delta}x}{yy^{\sigma}}
=\frac{\frac{-b(t)xy^{\sigma}}{x+y}y - \left( \frac{b(t)xy^{\sigma}}{x+y}
-c(t)y^{\sigma}\right)x}{yy^{\sigma}}
=\frac{-b(t)xy^{\sigma} + c(t)xy^{\sigma}}{yy^{\sigma}}=(c-b)(t)w,
\end{equation*}
which is a first-order linear dynamic equation with solution  
\begin{equation*}
w(t)=e_{c-b}(t,t_0)w_0,
\end{equation*}
i.e., 
\begin{equation}
\label{Eq3_B}
y(t) = \kappa e_{\ominus (c-b)}(t,t_0) x(t).
\end{equation}
We plug \eqref{Eq3_B} into \eqref{SIRT1} to get
\begin{equation*}
x^{\Delta}  = -\frac{b(t)xx^{\sigma}e_{\ominus (c-b)}(\sigma(t),t_0)\kappa}{x
+xe_{\ominus (c-b)}(t,t_0)\kappa}= -\frac{b(t)e_{\ominus (c-b)}(\sigma(t),
t_0)\kappa}{1+e_{\ominus (c-b)}(t,t_0)\kappa}x^{\sigma}
= - g(t) x^{\sigma},
\end{equation*}
which has the solution
\begin{equation*}
x(t) = e_{\ominus g}(t,t_0)x_0.
\end{equation*}
By \eqref{Eq3_B}, we obtain 
\begin{equation*}
y(t)= y_0e_{\ominus (g \oplus (c-b))}(t,t_0),
\end{equation*}
and thus,
\begin{equation*}
z(t) = N-x(t)-y(t)=N - e_{\ominus g}(t,t_0) 
\left( x_0 + y_0  e_{\ominus (c-b)}(t,t_0)  \right).
\end{equation*}
This shows that $x,y,z$ are as given in the statement. Conversely, 
it is easy to show that $x,y,z$ as given in the statement solve \eqref{SIRT1}.
The proof is complete.
\end{proof}

\begin{remark}
If $c-b \in \mathcal{R}^+$ and $x_0>0$, $y_0,z_0 \geq 0$, then $x,y,z \geq 0$ 
for all $t \in \mathbb{T}$. For $\mathbb{T}=\mathbb{R}$, this condition 
is satisfied, since $\mu(t)=0$ for all $t \in \mathbb{R}$.
\end{remark}

\begin{remark}
If $b(t)=c(t)$ for all $t \in \mathbb{T}$, then $c-b \in \mathcal{R}$, 
and, by Theorem~\ref{Thm1}, the solution of \eqref{SIRT1} is  
\begin{equation*}
\begin{cases}
x(t) = e_{\ominus g}(t,t_0)x_0,& \quad \\
y(t) = e_{\ominus g}(t,t_0)y_0,& \quad \\
z(t) = N - e_{\ominus g}(t,t_0) \left( x_0 + y_0 \right),& \quad
\end{cases}
\end{equation*}
where
\begin{equation*}
g(t) = \frac{b(t)\kappa}{1 + \kappa}.
\end{equation*}
\end{remark}

As an application of Theorem~\ref{Thm1}, we introduce the discrete epidemic model
\begin{equation}
\begin{cases}
\label{SIRZ1}
x(t+1)  =x(t) -\frac{b(t) x(t) y(t+1)}{x(t)+y(t)},& \quad \\[2mm]
y(t+1) = y(t) + \frac{b(t) x(t) y(t+1)}{x(t)+y(t)} - c(t) y(t+1),& \quad \\[2mm]
z(t+1)  = z(t) + c(t) y(t+1),& \quad
\end{cases}
\end{equation}
$t \in \mathbb{Z}$, with initial conditions 
$x(t_0)=x_0>0$, $y(t_0)=y_0>0$, $z(t_0)=z_0\geq 0$.
Note that the second equation of \eqref{SIRZ1} can be represented as
$$
y(t+1)  = \frac{1}{1+\delta(t)} y(t),
$$
which implies that a fraction, namely $\frac{1}{1+\delta}$, of the infected 
individuals remain infected. If the rate with which susceptibles are getting 
infected is higher than the rate with which infected are removed, i.e., 
$\varphi=b\frac{x}{x+y}>c$, then the multiplicative factor $\delta= c-\varphi$ 
is greater than one, else less than one. Slightly rewriting the first equation 
into the form 
$$
x(t+1) +\varphi(t) y(t+1)  = x(t)
$$
provides the interpretation that some susceptible individuals stay 
in the group of susceptibles, others become infected and contribute 
the fraction $\varphi$ to the group of infected. A similar inference 
can be drawn from 
$$
z(t+1)= z(t)+c(t)y(t+1).
$$
The number of removed individuals is the sum of the already removed individuals 
and a proportion of infected individuals that are removed 
at the end of the time step. 

The following theorem is a direct consequence of Theorem~\ref{Thm1}. 

\begin{theorem}
\label{Thm1Z}
If $1+c(t)-b(t), 1+g(t) \neq 0$ for all $t \in \mathbb{Z}$, where
\begin{equation*}
g(t) = \frac{b(t) \kappa }{\left[\prod_{i=t_0}^{t} 
(1+(c-b)(i))\right]+\kappa (1+(c-b)(t))},
\end{equation*}
then the unique solution to \eqref{SIRZ1} is given by
\begin{equation*}
\begin{cases}
x(t) = x_0\left[\prod_{i=t_0}^{t-1} \left(1+g(i)\right)\right]^{-1}, & \quad \\[2mm]
y(t) = y_0\left[\prod_{i=t_0}^{t-1} (1+(c-b)(i))(1+g(i))\right]^{-1}, & \quad \\[2mm]
z(t) = N -\left(x_0+y_0\left[\prod_{i=t_0}^{t-1} (1+(c-b)(i))\right]^{-1}\right)
\left[\prod_{i=t_0}^{t-1} \left(1+g(i)\right)\right]^{-1}, & 
\end{cases}
\end{equation*}
where $N=x_0+y_0+z_0$ and $\kappa= \frac{y_0}{x_0}$.
\end{theorem}

\begin{example}
Consider a disease with periodic transmission rate, for example due to sensitivity 
of bacteria to temperature or hormonal cycles. In this case, we might choose 
$b(t)=\frac{1}{2}+\frac{1}{4}\sin(mt)$ with $m\in \mathbb{R}\backslash \{0\}$. 
To account for medical advances, we let $c(t)=\frac{1}{t+1}$. Note that 
$1+c(t)\neq b(t)$ because $\frac{1}{4}\leq b(t) \leq \frac{3}{4}<1$ 
and $1+g(t) \neq 0$ for all $t \in \mathbb{Z}$. 
The solution is then given by Theorem~\ref{Thm1Z} with 
\begin{equation*}
g(t) =  \frac{ 2+\sin(mt)  }{\left[\frac{4}{\kappa}\prod_{i=t_0}^{t} 
\frac{3+i}{2(1+i)}-\frac{1}{4} \sin(m  i)\right]+\frac{2(3+t)}{1+t}-\sin(mt)}.
\end{equation*}
\end{example}

\begin{remark}
If $\mathbb{T}=\mathbb{R}$, $b,c \in \mathbb{R}$, $b \neq c$, and $t_0=0$, 
then, by Theorem~\ref{Thm1}, the solution to \eqref{SIRT1} is 
\begin{align*}
x(t) &= x_0e^{-\int_0^t g(s) \, {\rm d} s}=x_0e^{-b \int_0^t 
\frac{\kappa e^{-(c-b)s}}{ 1+ \kappa e^{-(c-b)s}} \, {\rm d}s}
=x_0e^{-\frac{b}{b-c} \left(\ln (1+ \kappa e^{-(c-b)t}) - \ln (1+ \kappa) \right)}\\
&= x_0\left( 1+ \kappa e^{-(c-b)t} \right)^{-\frac{b}{b-c}} 
\left(1+ \kappa \right)^{\frac{b}{b-c}},
\end{align*}
which is consistent with \eqref{Eq1_B}. If $c=b$, then Theorem~\ref{Thm1} 
provides the solution as 
\begin{equation*}
x(t) = x_0 e^{-\int_0^t g(s) \, ds}=x_0e^{-b \int_0^t 
\frac{\kappa }{ 1+ \kappa } \, ds}=x_0 e^{\frac{-b \kappa}{1+\kappa}t},
\end{equation*}
which is consistent with the continuous results.
\end{remark}

\begin{example}
\label{ex:19}
Let us consider the SIR model \eqref{SIRT1} with
\begin{equation*}
b = 0.4, \quad c = 0.2, \quad x_0 = 0.8, \quad y_0 = 0.2, \quad z_0 = 0.
\end{equation*}
\begin{center}
\begin{figure}
\begin{subfigure}{.5\textwidth}
\centering
\includegraphics[width=.9\linewidth]{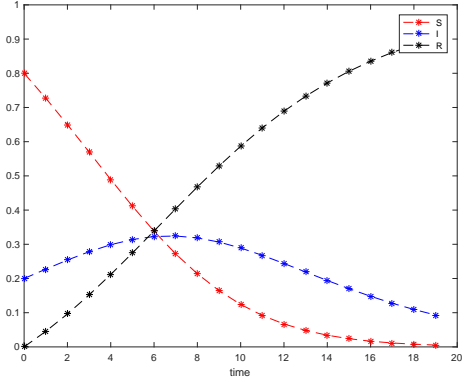}
\captionsetup{width=0.48\textwidth}
\centering
\caption{$\mathbb{T}=\mathbb{Z}$}
\label{fig:1a}
\end{subfigure}%
\begin{subfigure}{.5\textwidth}
\includegraphics[width=.9\linewidth]{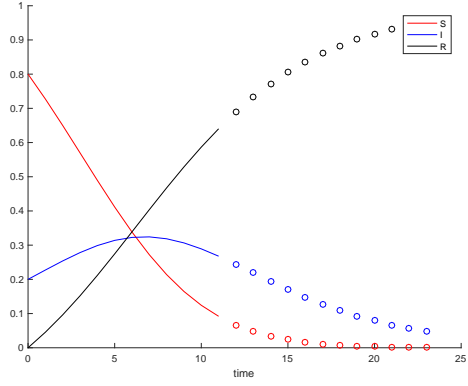}
\captionsetup{width=0.48\textwidth}
\centering
\caption{$\mathbb{T}=[0,12]\cup \{13,14,15,\ldots, 24\}$}
\label{fig:1b}
\end{subfigure}
\captionsetup{width=0.8\textwidth}
\centering
\caption{the $x$=Susceptible ($S$), $y$=Infected ($I$), and $z$=Removed ($R$) 
dynamics of Example~\ref{ex:19}.}
\label{example1}
\end{figure}
\end{center}
In Figure~\ref{fig:1a}, we show the solution in the discrete-time case $\mathbb{T}=\mathbb{Z}$ 
determined by \eqref{SIRZ1}; in Figure~\ref{fig:1b}, we plot the solution to \eqref{SIRT1} 
for the partial continuous, partial discrete time scale $\mathbb{T}=[0,12]\cup \{13,\ldots,24\}$.
\end{example}


\section{Long Term Behavior}

We start this section by recalling the following results. 

\begin{lemma}[See \protect{\cite[Lemma 3.2]{MR2842561}}]
\label{AddBL1}
If $p \in \mathcal{R}^+$, then 
\begin{equation*}
0 < e_p(t,t_0) \leq \exp \left\{ \int_{t_0}^t p(\tau) \Delta \tau \right \}
\quad \quad \mbox{for all } t \geq t_0.
\end{equation*}
\end{lemma}

\begin{lemma}[See \protect{\cite[Remark 2]{MR2145447}}]
\label{AddBL2}
If $p\in {\rm C}_{\rm rd}$ and $p(t)\geq 0$ for all $t \in \mathbb{T}$, then 
\begin{equation*}
1 + \int_{t_0}^t p(\tau) \Delta \tau \leq e_{p}(t,t_0) \leq \exp 
\left \{ \int_{t_0}^t p(\tau) \Delta \tau \right \} 
\quad \quad \mbox{for all } t \geq t_0.
\end{equation*}
\end{lemma}

The equilibriums of \eqref{SIRT1} are given as follows. 

\begin{lemma}
\label{AddBL3}
Suppose $c(t) >0$ for some $t \in \mathbb{T}$. The equilibriums of \eqref{SIRT1} 
are given by the plane $(\alpha, 0, N-\alpha)$, where $\alpha \in [0,N]$ and $N=x_0+y_0+z_0$.
\end{lemma}

\begin{proof}
Assume $x,y,z$ are constant solutions of \eqref{SIRT1}. 
Then, $0=z^{\Delta}(t)=c(t)y(t)$, 
so $y(t)=0$ for all $t\in \mathbb{T}$. Therefore,
\begin{equation*}
-\frac{b(t)xy}{x+y} = x^{\Delta}=0=y^{\Delta}
=\frac{b(t)xy}{x+y}-c(t)y \quad \mbox{ for any } 0 \leq x \leq N
\end{equation*}
and the proof is complete.
\end{proof}

\begin{theorem}
\label{AddBT1}
Consider \eqref{SIRT1} and assume $\mathbb{T}$ is unbounded from above. 
Assume $b,c: \mathbb{T} \to \mathbb{R}_0^+$, $c-b \in \mathcal{R}^+$, 
$x_0, y_0 >0$, and $z_0 \geq 0$. Moreover, assume 
\begin{equation}
\label{ProofBEq1}
\exists \, L >0: \, \int_{t_0}^t (c-b)(\tau) \, \Delta \tau \leq L 
\quad \quad \mbox{ for all } t \geq t_0
\end{equation}
and 
\begin{equation}
\label{ProofBEq2}
\int_{t_0}^{\infty} \frac{b(\tau)}{1+\mu(\tau)(c-b)(\tau)} 
\, \Delta \tau = \infty.
\end{equation}
Then all solutions of \eqref{SIRT1} converge to the equilibrium 
$(0,0,N)$, where $N=x_0+y_0+z_0$.
\end{theorem}

\begin{proof}
By Lemma~\ref{AddBL1}, 
\begin{equation*}
0 < e_{c-b}(t,t_0) \leq e^{\int_{t_0}^t (c-b)(\tau) \, \Delta \tau} 
\stackrel{\eqref{ProofBEq1}}{\leq} e^{L}, \quad t\geq t_0.
\end{equation*}
By Lemma~\ref{AddBL2}, since $g \geq 0$, we get
\begin{align*}
e_g(t,t_0) & \geq 1+ \int_{t_0}^t g(\tau) \, \Delta \tau= 1+\int_{t_0}^t 
\frac{b(\tau)}{1+\mu(t)(c-b)(\tau)}\frac{\kappa}{\kappa+e_{c-b}(\tau,t_0)} \, \Delta \tau\\
&\geq 1+ \frac{\kappa}{\kappa+e^{L} } \int_{t_0}^t \frac{b(\tau)}{1+\mu(t)(c-b)(\tau)} 
\, \Delta \tau \stackrel{\eqref{ProofBEq2}}{\to} \infty, \quad t \to \infty.
\end{align*}
Then $e_{\ominus g}(t,t_0) \to 0$ as $t \to \infty$, so that 
\begin{equation*}
\lim_{t \to \infty} x(t)=0, \quad \lim_{t \to \infty} y(t)=0, 
\quad \lim_{t \to \infty} z(t)=N
\end{equation*}
due to Theorem~\ref{Thm1}.
\end{proof}

\begin{corollary} 
If $b(t)=c(t)$ for all $t \in \mathbb{T}$, 
then the conclusion of Theorem~\ref{AddBT1} holds provided 
\begin{equation*}
\int_{t_0}^{\infty} b(\tau) \, \Delta \tau = \infty.
\end{equation*}
\end{corollary}

\begin{corollary} 
If $b$ and $c$ are constants, then the conclusion of 
Theorem~\ref{AddBT1} holds provided
\begin{equation*}
c-b \in \mathcal{R}^+ \quad \mbox{and } \quad b \geq c.
\end{equation*} 
\end{corollary}

\begin{theorem}
\label{AddBT2}
Consider \eqref{SIRT1} and assume $\mathbb{T}$ is unbounded from above. 
Assume $b,c: \mathbb{T} \to \mathbb{R}_0^+$, $c-b \in \mathcal{R}^+$, 
$x_0, y_0 >0$, and $z_0 \geq 0$. Moreover, assume 
\begin{equation}
\label{ProofBEq3}
\exists \, M >0: \, b(t)\leq M(c-b)(t) 
\quad \quad \mbox{ for all } t \in \mathbb{T}
\end{equation}
and 
\begin{equation}
\label{ProofBEq4}
\int_{t_0}^{\infty} (c-b)(\tau)\, \Delta \tau = \infty.
\end{equation}
Then all solutions of \eqref{SIRT1} converge to the equilibrium 
$(\alpha, 0, N-\alpha)$ for some $\alpha \in (0,N)$.
\end{theorem}

\begin{proof}
Note first that \eqref{ProofBEq3} implies $c(t)\geq b(t)$ 
for all $t \in \mathbb{T}$. By Lemma~\ref{AddBL2}, we have 
\begin{equation*}
e_{c-b}(t,t_0) \geq 1+ \int_{t_0}^t (c-b) (\tau) \, 
\Delta \tau \stackrel{\eqref{ProofBEq4}}{\to} \infty, \quad t \to \infty, 
\end{equation*}
so 
\begin{equation}
\label{ProofBEq5}
\lim_{t \to \infty} e_{\ominus (c-b)} (t,t_0) = 0.
\end{equation}
Next, 
\begin{align*}
g(t)&=\frac{b(t)}{1+\mu(t)(c-b)(t)} \frac{\kappa}{
\kappa+e_{c-b}(t,t_0)}\leq \frac{b(t)}{1+\mu(t)(c-b)(t)}  
\frac{\kappa}{e_{c-b}(t,t_0)} = \frac{\kappa b(t)}{e_{c-b}(\sigma(t),t_0)}\\ 
&\leq \frac{\kappa M(c-b)(t)}{e_{c-b}(\sigma(t),t_0)},
\end{align*}
and thus, using \cite[Theorem 2.39]{MR1843232}, we get 
\begin{equation*}
\int_{t_0}^t g(\tau) \, \Delta \tau \leq M\kappa  
\int_{t_0}^t \frac{(c-b)(\tau)}{e_{c-b}(\sigma(\tau), t_0)} \, \Delta \tau 
=  M\kappa \left[ 1- \frac{1}{e_{c-b}(t,t_0)} \right] < M\kappa.
\end{equation*}
By Lemma~\ref{AddBL2}, since $g\geq 0$ for all $t \in \mathbb{T}$, we get 
\begin{equation*}
1\leq 1+\int_{t_0}^t g(\tau) \, \Delta \tau \leq e_g(t,t_0) 
\leq \exp\left \{ \int_{t_0}^t g(\tau) \, \Delta \tau \right \} 
< e^{\kappa M} \quad \mbox{ for all } t \geq t_0,
\end{equation*} 
so $\lim_{t \to \infty}e_g(t,t_0)$ exists and is bounded from below 
by $1$ and bounded from above by $e^{\kappa M}$.
We therefore get that $\lim_{t \to \infty}e_{\ominus g}(t,t_0)$ exists 
and is greater than or equal to $e^{-\kappa M}>0$. Hence,
\begin{equation*}
\alpha := \lim_{t \to \infty} x(t)>0, \quad \lim_{t \to \infty} y(t)=0, 
\quad \lim_{t \to \infty} z(t)=N-\alpha
\end{equation*}
due to \eqref{ProofBEq5} and Theorem~\ref{Thm1}.
\end{proof}

\begin{corollary}
If $b$ and $c$ are constants, then the conclusion of 
Theorem~\ref{AddBT2} holds provided 
\begin{displaymath}
b<c.
\end{displaymath}
\end{corollary}

Finally, we give a result that describes 
the monotone behavior of the solution $y$.

\begin{theorem}
\label{AddBT3}
If $c(t) \geq b(t)$ for all $t \in \mathbb{T}$ or $\frac{x_0}{x_0+y_0}b(t)
\leq c(t)\leq b(t)$ for all $t \in \mathbb{T}$, then $y$ is decreasing. 
If $\frac{x_0}{x_0+y_0}b(t_0) \geq c(t_0)$, then $y^{\Delta}(t_0)\geq 0$. 
\end{theorem}

\begin{proof}
If $\frac{x_0}{x_0+y_0}b(t_0)> c(t_0)$, then 
\begin{equation*}
y^{\Delta}(t_0) = \frac{b(t_0) x(t_0)y(\sigma(t_0))}{x(t_0)+y(t_0)}
-c(t_0)y(\sigma(t_0))= 
\left[ b(t_0)\frac{x_0}{x_0+y_0} -c(t_0) 
\right]y(\sigma(t_0))\geq 0.
\end{equation*}
If $c(t)\geq b(t)$ for all $t \in \mathbb{T}$, then 
\begin{align*}
y^{\Delta}(t)&=\frac{b(t)x(t)y(\sigma(t))}{x(t)+y(t)}-c(t)y(\sigma(t))
\leq \frac{b(t)x(t)y(\sigma(t))}{x(t)+y(t)}-b(t)y(\sigma(t))\\[2mm]
&=b(t)y(\sigma(t))\left[ \frac{x(t)}{x(t)+y(t)}-1\right]
= -\frac{b(t)y(t)y(\sigma(t))}{x(t)+y(t)}\leq 0 
\quad \mbox{for all } t\in \mathbb{T}.
\end{align*}
Next, we calculate 
\begin{equation*}
\left(   \frac{x}{x+y}\right)^{\Delta} 
= \frac{x^{\Delta}y- y^{\Delta}x}{(x+y)(x^{\sigma}+y^{\sigma})}
= \frac{-\frac{b(t)xy^{\sigma}}{x+y}y - \left( \frac{b(t)xy^{\sigma}}{x+y} 
- c(t)y^{\sigma}  \right)x   }{(x+y)(x^{\sigma}+y^{\sigma})}
= \frac{(c-b)(t)xy^{\sigma}}{(x+y)(x^{\sigma}+y^{\sigma})}.
\end{equation*}
If $\frac{x_0}{x_0+y_0}b(t)\leq c(t)\leq b(t)$ for all $t \in \mathbb{T}$, then 
\begin{align*}
y^{\Delta}(t)&= \frac{b(t)x(t)y(\sigma(t))}{x(t)+y(t)}-c(t)y(\sigma(t)) \leq 
b(t)y(\sigma(t))\frac{x_0}{x_0+y_0} - c(t)y(\sigma(t))\\
&= \left[ b(t)\frac{x_0}{x_0+y_0} - c(t) \right] y(\sigma(t))\leq 0 
\quad \mbox{ for all } t \in \mathbb{T}.
\end{align*}
This completes the proof.
\end{proof}

\begin{example}
\label{ex:29}
For $\mathbb{T}=\mathbb{Z}$, $S(0)=0.8, I(0)=0.1, R(0)=0.1$, and $b=0.2$, 
we get for $c=0.3$ the limit behavior for the solutions as shown in 
Figure~\ref{Figlimit1}. Changing $c$ to $0.1$ such that $b>c$, we get 
the behavior demonstrated in Figure~\ref{Figlimit2}.
\begin{center}
\begin{figure}
\begin{subfigure}{.5\textwidth}
\centering
\includegraphics[width=.9\linewidth]{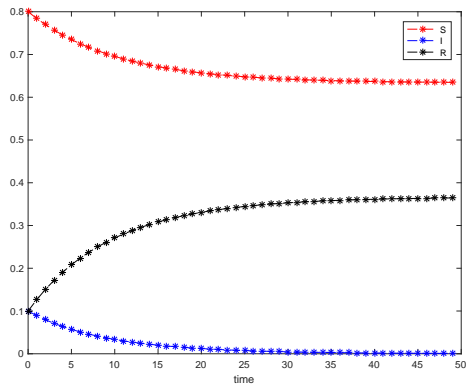}
\captionsetup{width=0.48\textwidth}
\centering
\caption{$b=0.2<0.3=c$}\label{Figlimit1}
\end{subfigure}%
\begin{subfigure}{.5\textwidth}
\includegraphics[width=.9\linewidth]{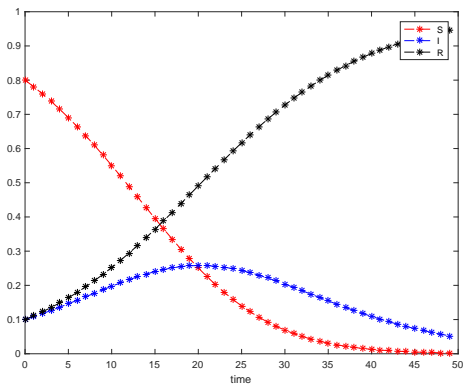}
\captionsetup{width=0.48\textwidth}
\centering
\caption{$b=0.2>0.1=c$}\label{Figlimit2}
\end{subfigure}
\captionsetup{width=0.8\textwidth}
\centering
\caption{the $x$=Susceptible ($S$), $y$=Infected ($I$), and $z$=Removed ($R$) 
long term behavior of Example~\ref{ex:29}.}
\label{example2}
\end{figure}
\end{center}
\end{example}


\subsection*{Acknowledgement}

Torres has been partially supported by FCT within CIDMA project
UID/MAT/04106/2019, and by TOCCATA FCT project PTDC/EEI-AUT/2933/2014.
The authors are very grateful to three anonymous reviewers 
for several constructive comments, questions and suggestions,
which helped them to improve the paper.



\end{document}